\documentclass[12pt,leqno]{amsart}
\usepackage{latexsym,amsmath,amssymb,xcolor}
\usepackage{graphicx,hyperref}

\title[A metric implicit function theorem]{An implicit function theorem for Lipschitz mappings into metric spaces}
\author{Piotr Haj{\l}asz and Scott Zimmerman}

\address{P.\ Haj{\l}asz: Department of Mathematics, University of Pittsburgh, 301
  Thackeray Hall, Pittsburgh, PA 15260, USA, {\tt hajlasz@pitt.edu}}

\address{S. Zimmerman: Department of Mathematics, University of Connecticut, 341 Mansfield Road U1009, 
  Storrs, CT 06269, USA, {\tt scott.zimmerman@uconn.edu}}

\thanks{P.H.\ was supported by NSF grant DMS-1800457.}

plus 6pt minus 12pt
\def\rank{{\rm rank\,}}

\newcommand{\overbar}[1]{\mkern 1.7mu\overline{\mkern-1.7mu#1\mkern-1.5mu}\mkern 1.5mu}

\def\eps{\varepsilon}

\def\H{{\mathcal H}}
\def\K{{\mathcal K}}

\newtheorem{theorem}{Theorem}
\newtheorem{lemma}[theorem]{Lemma}
\newtheorem{corollary}[theorem]{Corollary}
\newtheorem{proposition}[theorem]{Proposition}

\def\diam{{\rm diam\,}}
\def\dist{{\rm dist\,}}

\def\lip{{\rm Lip\,}}

\theoremstyle{definition}
\newtheorem{remark}[theorem]{Remark}

\newcommand{\barint}{
\rule[.036in]{.12in}{.009in}\kern-.16in \displaystyle\int }

\newcommand{\barcal}{\mbox{$ \rule[.036in]{.11in}{.007in}\kern-.128in\int $}}

\newcommand{\bbbn}{\mathbb N}

\newcommand{\bbbr}{\mathbb R}

\def\ap{\operatorname{ap}}

\def\diam{\operatorname{diam}}

\def\dist{\operatorname{dist}}

\def\mvint_#1{\mathchoice
          {\mathop{\vrule width 6pt height 3 pt depth -2.5pt
                  \kern -8pt \intop}\nolimits_{\kern -3pt #1}}%
          {\mathop{\vrule width 5pt height 3 pt depth -2.6pt
                  \kern -6pt \intop}\nolimits_{#1}}%
          {\mathop{\vrule width 5pt height 3 pt depth -2.6pt
                  \kern -6pt \intop}\nolimits_{#1}}%
          {\mathop{\vrule width 5pt height 3 pt depth -2.6pt
                  \kern -6pt \intop}\nolimits_{#1}}}

\numberwithin{theorem}{section} \numberwithin{equation}{section}

\begin{document}

\sloppy


\begin{abstract}
We prove a version of the implicit function theorem for Lipschitz mappings 
$f:\mathbb{R}^{n+m}\supset A \to X$ into arbitrary metric spaces.
As long as the 
pull-back of the Hausdorff content $\mathcal{H}_{\infty}^n$ by $f$ has positive upper $n$-density
on a set of positive Lebesgue measure,
then, there is a local diffeomorphism $G$ in $\bbbr^{n+m}$ and a Lipschitz
map $\pi:X\to \bbbr^n$ such that
$\pi\circ f\circ G^{-1}$, when restricted to a certain 
subset of $A$ of positive measure, is the orthogonal projection of $\bbbr^{n+m}$ onto the first $n$-coordinates.
This may be seen as a qualitative version of a simlar result of Azzam and Schul \cite{AzzSch}.
The main tool in our proof is the metric change of variables introduced in \cite{HajMal}.
\end{abstract}

\maketitle

\centerline{\emph{In memoriam: William P. Ziemer (1934-2017)}}

\section{Introduction}

The classical implicit function theorem (IFT)
ensures that the  map is
structurally very nice near points where the derivative of the map
has a certain rank.
In this paper, we present a version of the IFT for Lipschitz mappings 
$f:\mathbb{R}^{n+m} \supset A \to X$ into arbitrary metric spaces.
It turns out that in the case of mappings into metric spaces, the upper density defined below will play a role of the Jacobian of $f$.
For a measurable set $A \subset \bbbr^{k}$, and $x\in A$, we
define the \emph{lower} and \emph{upper} $n$-\emph{densities} of a mapping 
$f:A \to X$ as
$$
\Theta^{*n}(f,x) 
:= 
\limsup_{r \to 0}\frac{\H_{\infty}^n(f(B(x,r) \cap A))}{\omega_n r^n},
\quad
\Theta_*^n(f,x)
:= 
\liminf_{r \to 0}\frac{\H_{\infty}^n(f(B(x,r) \cap A))}{\omega_n r^n}.
$$
These are simply the upper and the lower $n$-densities of the pull-back of $\mathcal{H}_{\infty}^n$ by $f$ on $A$.
Here $\omega_n$ is the volume of the unit ball in $\bbbr^n$ and 
the $\H_\infty^n$ is the {\em Hausdorff content} defined for subsets of $X$ by
$$
\H^n_\infty(E)=\inf \frac{\omega_n}{2^n}\sum_{i=1}^\infty (\diam A_i)^n,
$$
where the infimum is taken over all coverings of $E$,
i.e. $E\subset\bigcup_{i=1}^\infty A_i$.
Note that the Hausdorff content of {\em any} bounded set is finite, and, for an $L$-Lipschitz map $f:A\to X$, $\Theta^{*n}(f,x)\leq L^n$ for {\em all} $x\in A$.

The reader may want to compare these definitions with the definition (and properties) of the upper and lower densities of measures in \cite{ambrosiok,mattila,simon}.

The following observation will be useful throughout the paper:
\begin{equation}
\label{1410}
\Theta^{*n}(f,x)=0
\quad
\text{if and only if}
\quad
\lim_{d\to 0}\frac{\H_{\infty}^n(f(Q(x,d) \cap A))}{\omega_n d^n} = 0
\end{equation}
where $Q(x,d)$ is the cube centered at $x$ with side length $d$.
(Here and in what follows, a cube has edges parallel to the coordinate axes.)
The main result of the paper is as follows:
\begin{theorem}[Metric IFT]
\label{HardSard2}
Fix a metric space $X$, 
a set $A \subset \bbbr^{n+m}$ with positive Lebesgue measure, 
and a Lipschitz mapping $f:A \to X$. 
Suppose $\Theta^{*n}(f,x) > 0$ on a subset of $A$ with positive Lebesgue measure.
Then 
\begin{itemize}
\item[(A)] $\H^n(f(A))>0$;
\item[(B)] 
There is a set $K \subset A$ with positive Lebesgue measure, a 
bi-Lipschitz
$C^1$-diffeomorphism $G:U\to G(U)\subset\bbbr^{n+m}$ defined on an open set $U \supset K$ and a $\sqrt{n}$-Lipschitz map
$\pi: X\to\bbbr^n$ such that
$$
\pi\circ f\circ G^{-1}(x_1,\ldots,x_n,y_{1}\ldots,y_m)=(x_1,\ldots,x_n)
\quad\text{for all $(x,y)\in G(K)$}
$$
is a projection on the first $n$ coordinates when restricted to the set $G(K)$.
\end{itemize}
Moreover the mapping $F=f\circ G^{-1}$ defined on $G(K)$ satisfies
\begin{itemize}
\item[(C)] 
$F^{-1}(F(x,y)) \cap G(K) \subset \{x\} \times \bbbr^m$ for any $(x,y) \in G(K)$;
\item[(D)]
$F|_{(\bbbr^n \times \{y\}) \cap G(K)}$ is bi-Lipschitz for any $y \in \bbbr^m$. 
\end{itemize}
\end{theorem}
\begin{remark}
It follows from the proof that we can exhaust the set of points where $\Theta^{*n}(f,x)>0$ by sets $K$ as in (B) up to a set of $\H^{n+m}$ measure zero.
(See the application of Lemma~\ref{federer} in the proof of Lemma~\ref{3.1} as well as Remark~\ref{bluerem}.)
\end{remark}
\begin{remark}
\label{R145}
The map $\pi:X\to\bbbr^n$ is in fact $1$-Lipschitz as a map from $X$ to $(\bbbr^n,\ell^\infty_n)$ where the norm $\ell^\infty_n$ is defined by $\Vert (x_1,\ldots,x_n)\Vert_\infty=\max_i|x_i|$. This will follow from our proof.
\end{remark}
\begin{remark}
Statement (C) means that the preimage under $F$ of any point in $F(G(K))=f(K)$ is contained in an $m$-dimensional subspace of $\bbbr^{n+m}$ orthogonal to $\bbbr^n$.
For related results about the structure of preimages $f^{-1}(z)$ of Lipschitz maps,
see \cite[Theorem~1.2]{karmanova}, \cite[Theorem~4.16]{reichel}.
\end{remark}
\begin{remark}
In fact, we will prove a quantitative lower bound in (D):
\begin{equation}
\label{404}
\Vert x_1-x_2\Vert_\infty\leq d(F(x_1,y),F(x_2,y))
\end{equation}
for any $y\in\bbbr^m$ and all $(x_1,y),(x_2,y)\in (\bbbr^n \times \{y\}) \cap G(K)$.
\end{remark}
\begin{remark}
The classical implicit function theorem is stated using a condition about the rank of the derivative of $f$, and the condition $\Theta^{*n}(f,x)>0$ is a related one. 
Indeed, in the case $X = \mathbb{R}^n$, we will see in Proposition~\ref{euclem} that the Jacobian of $f$
defined by
$|J^nf|(x) = \sqrt{\det (Df)(Df)^T(x)}$ satisfies 
$\Theta^{*n}(f,x) = |J^n f|(x)$  almost everywhere. See also Lemma~\ref{elllem} for the case of mappings
$f:A\to\ell^\infty$.
\end{remark}
\begin{remark}
In the theorem we cannot replace the density condition $\Theta^{*n}(f,x)>0$ by the simpler measure condition $\H^n(f(A))>0$. Indeed, even in the Euclidean case, Kaufmann \cite{kaufman} constructed a surjective $C^1$ mapping
$f:\bbbr^{n+1}\to\bbbr^n$, $n\geq 2$, satisfying $\rank Df\leq 1$ {\em everywhere}. For such a map, condition (B)
cannot be satisfied since it would imply that $\rank Df\geq n$ on $K$.
\end{remark}

Recall that a set $E\subset\bbbr^{n+m}$ is {\em countably $\H^m$-rectifiable}
if there are Lipschitz mappings $f_i:\bbbr^m\supset E_i\to\bbbr^{n+m}$, $i\in\mathbb{N}$, such that $\H^m(E\setminus\bigcup_{i=1}^\infty f(E_i))=0$.
As a corollary of Theorem~\ref{HardSard2} we obtain
\begin{corollary}
\label{HSC}
Fix a metric space $X$, 
a set $A \subset \bbbr^{n+m}$ with positive Lebesgue measure, 
and a Lipschitz mapping $f:A \to X$. 
Suppose $\Theta^{*n}(f,\cdot) > 0$ almost everywhere in $A$.
Then $f^{-1}(x)$ is countably $\H^m$-rectifiable for $\H^n$-almost all $x\in X$.
\end{corollary}
See Section~\ref{HSCProof} for the proof.
For related results see \cite[Theorem~1.2]{karmanova}, \cite[Theorem~4.16]{reichel}.

Our result may be seen as a qualitative version of a theorem proven in 2012 by 
Azzam and Schul \cite{AzzSch}.
In that paper, the authors proved the following quantitative version of the IFT
for Lipschitz mappings into metric spaces:

\begin{theorem}[Quantitative metric IFT; Azzam and Schul, 2012]
\label{AS}
Fix a metric space $X$
and a $1$-Lipschitz mapping
$f:\bbbr^{n+m} \to X$. 
Suppose 
$0 < \H^n(f([0,1]^{n+m})) \leq 1$
and
\begin{equation}
\label{HCap}
0< \delta \leq \mathcal{H}_{\infty}^{n,m}(f,[0,1]^{n+m})
\end{equation}
for some $\delta > 0$.
Then there are constants $\Lambda=\Lambda(n,m,\delta)>1$ and $\eta=\eta(n,m,\delta)>0$,
a set $K \subset[0,1]^{n+m}$ with 
\begin{equation}
\label{KBound}
\mathcal{H}^{n+m}(K) \geq \eta,
\end{equation}
and a $\Lambda$-bi-Lipschitz homeomorphism $G:\bbbr^{n+m} \to \bbbr^{n+m}$
such that $F = f \circ G^{-1}$ satisfies
$$
F^{-1}(F(x,y)) \cap G(K) \subset \{x\} \times \bbbr^m
\quad \text{for any } (x,y)\in G(K) \subset \mathbb{R}^{n+m}
$$
and $F|_{ (\bbbr^n \times \{y\}) \cap G(K) }$ is $\Lambda$-bi-Lipschitz for any $y \in \bbbr^m$. 
\end{theorem}
The authors of \cite{AzzSch} call $\H_{\infty}^{n,m}$ the $(n,m)$-{\em Hausdorff content of} $f$.
It is defined for a Lipschitz map $f:Q\to X$ from a cube $Q\subset\bbbr^{n+m}$ to a metric space by
\begin{equation}
\label{AS12}
\H_\infty^{n,m}(f,Q)=
\inf \sum_{j=1}^\infty \H^n_\infty(f(Q_j))d_j^m,
\end{equation}
where the infimum is taken over all families of open pairwise disjoint cubes 
$Q_j\subset Q$ of side length $d_j$ that cover $Q$ up to a set of measure zero.

Note that Theorems~\ref{HardSard2} and \ref{AS} 
provide the same qualitative structure on the vertical and horizontal slices
of the preimage of $F$.
However, Theorem~\ref{AS} is a {\em quantitative} version of the metric IFT in the sense that it
provides the lower bound \eqref{KBound} 
which depends only on the dimensions $m$, $n$
and $\delta$  from \eqref{HCap}. Moreover, the mapping $G$ 
is a \emph{globally} defined $C$-bi-Lipschitz homomorphism
where $C$ depends only on $m$, $n$, and $\delta$.
Our result (Theorem~\ref{HardSard2}) does not contain these quantitative conclusions.
This is because the assumption \eqref{HCap} in Theorem~\ref{AS} 
is much stronger 
than the assumption 
that $\Theta^{*n}(f,x) > 0$ on a set of positive measure.
Indeed, Proposition~\ref{elllem2} shows that the positivity of $\Theta^{*n}(f,x)$
follows from the assumption \eqref{HCap}.
In fact, for any $\eps>0$,
one may construct a mapping $f:[0,1]^2 \to \mathbb{R}$
with $\Theta^{*1}(f,x) = 1$ almost everywhere
so that 
the set $K\subset\bbbr^2$ satisfying the conclusion of Theorem~\ref{HardSard2} 
(for a global bi-Lipschitz homeomorphism $G$) must 
satisfy $\mathcal{H}^{2}(K) < \eps$
(and hence \eqref{KBound} cannot hold).
See Proposition~\ref{fold} for the construction and a detailed statement.

On the other hand, while the assumptions of Theorem~\ref{HardSard2} are much weaker than those of Theorem~\ref{AS},
some of the conclusions seem stronger: (1) As we already pointed out, the condition about positivity
of $\Theta^{*n}(f,x)$ is much weaker than condition \eqref{HCap}; (2) Azzam and Schul assume that
$0<\H^n(f([0,1]^{n+m}))\leq 1$ while we do not assume anything about the Hausdorff measure of the image. In fact,
we prove the lower bound $\H^n(f(A))>0$ in (A) and finiteness of the measure of the image plays no role in our theorem; (3) Our mapping $G$ is a bi-Lipschitz $C^1$ diffeomorphism while their mapping $G$ is only a bi-Lipschitz map. However, their map is defined globally and ours is defined locally only; (4) While parts (C) and (D) are the same as the corresponding statements in Theorem~\ref{AS}, part (B) seems stronger than that. (C) and (D) easily follow from (B), but we do not know if (B) can be concluded from Theorem~\ref{AS};
(5) We obtain the quantitative lower bound estimate \eqref{404};
(6) At last, but not least, our proof is {\em much} simpler than that in \cite{AzzSch}.

The classical IFT states that a $C^1$ mapping has a nice structure near a point where the derivative has rank of a certain order. However, the classical IFT does not provide {\em any} estimate for the size of the set where the map is nice. Our result has the same feature as the classical one: we do not obtain any estimate for the size of the set $K$ except that it has a positive measure.

The main tool in the proof of Theorem~\ref{HardSard2}
will be the metric change of variables introduced in \cite{HajMal}.
This change of variables has been used to prove versions 
of Sard's theorem for Lipschitz mappings and BLD mappings into metric spaces \cite{HajMal,HajMalZim}.

This paper is organized as follows. In Section~\ref{prelim}
we collect basic definitions and lemmata needed in the proofs of Theorem~\ref{HardSard2}
and Corollary~\ref{HSC}.
In Sections~\ref{SecProof} and~\ref{HSCProof} we prove Theorem~\ref{HardSard2} and Corollary~\ref{HSC} respectively.
Finally, in Section~\ref{SecComp}, we prove some other results that help us compare Theorems~\ref{HardSard2} and~\ref{AS}, we prove that the condition $\H^{n,m}_\infty(f,Q)>0$ implies positivity of $\Theta^{*n}(f,x)$ on a set of positive measure (Proposition~\ref{elllem2}), 
we prove that, if 
$f:\bbbr^{n+m}\supset A\to\bbbr^n$ is Lipschitz, then 
$\Theta^{*n}(f,x)=\Theta_*^n(f,x)=|J^nf|(x)$ almost everywhere in $A$ (Proposition~\ref{euclem}), and 
we construct an example showing that we cannot obtain any lower bound for $\H^{n+m}(K)$ (Proposition~\ref{fold}).

Notation used in the paper is fairly standard. The $n$-dimensional Hausdorff measure will be denoted by $\H^n$. Note that in $\bbbr^n$, $\H^n$ equals the Lebesgue measure and we will use Hausdorff measure notation in place of the Lebesgue measure. Occasionally we will write $|E|$ to denote the Lebesgue measure of $E$. Notation
$\H^n_\infty$ will stand for the Hausdorff content defined above. The constant $\omega_n$ denotes the measure of the unit ball in $\bbbr^n$. The Banach space of bounded {\em real valued} sequences will be denoted by $\ell^\infty$.  Balls in metric spaces are denoted by $B(x,r)$, and $Q(x,d)$ denotes the Euclidean cube centered at $x$ with side length $d$. 
All cubes are assumed to have edges parallel to the coordinate axes. 
Occasionally a $k$-dimensional ball in a Euclidean space will be denoted by $B^k(x,r)$.
By a $\Lambda$-bi-Lipschitz homeomorphism $f:(X,d)\to (Y,\rho)$ we mean a homeomorphism satisfying $\Lambda^{-1}d(x,y)\leq \rho(f(x),f(y))\leq \Lambda d(x,y)$.
The tangent space to $\bbbr^k$ at $x\in\bbbr^k$ will be denoted by $T_x\bbbr^k$.
By $C$ we will denote a general constant whose value may change in a single string of estimates. Writing $C=C(n,m)$, for example, indicates that the constant $C$ depends on $n$ and $m$ only.

\noindent
{\bf Acknowledgements.} The authors would like to thank the referee for valuable comments that led to an improvement of the paper.

\section{Preliminaries}
\label{prelim}
In this section we collect basic definitions and results that will be used later on.

If $k>n$, then the $\H^n_\infty$ content of subsets of $\bbbr^k$ is very different from their Hausdorff measure. For example $\H^n_\infty(E)<\infty$ for any bounded set $E\subset\bbbr^k$, but
$\H^n(B)=\infty$ for any $k$-ball $B\subset\bbbr^k$. However, we have (see \cite[Theorem~2.6]{simon})
\begin{lemma}
\label{h=h}
$\H^n_\infty(E)=\H^n(E)$ for all sets $E\subset\bbbr^n$.
\end{lemma}
\begin{lemma}
\label{kura}
Every separable metric space admits an isometric embedding into $\ell^\infty$.
\end{lemma}
Indeed, given $x_0\in X$ and a dense set $\{x_i\}_{i=1}^\infty$ in a separable metric space $(X,d)$, 
$$
X\ni x \mapsto \kappa(x)=(d(x,x_i)-d(x_i,x_0))_{i=1}^\infty\in\ell^\infty
$$ 
is an isometric embedding.
This is the well known Kuratowski embedding for metric spaces.

For a proof of the following elementary result, see \cite[Corollary~4.1.7]{HKST}.
\begin{lemma}
\label{2.7}
Let $Y$ be a metric space, let $E\subset Y$ and let $f:E\to\ell^\infty$ be an $L$-Lipschitz mapping. Then there is an $L$-Lipschitz mapping $F:Y\to\ell^\infty$ such that $F|_E=f$.
\end{lemma}
The idea of the proof is very simple.
Each component $f_i$ of $f$ is $L$-Lipschitz and we define $F$ by extending each of the components of $f$ using the formula from the McShane extension. Then it is easy to verify that the resulting map is $L$-Lipschitz and it takes values in $\ell^\infty$.

Fix an integer $k \geq 1$,
and suppose $A \subset \mathbb{R}^k$ is measurable.
Recall that a function $f:A\to\bbbr$ is {\em approximately differentiable} at $x\in A$ if there is 
a measurable set $A_x\subset A$ and a linear map $L:\bbbr^n\to\bbbr$ such that $x$ is a density point of $A_x$ and 
$$
\lim_{A_x\ni y\to x}\frac{|f(y)-f(x)-L(y-x)|}{|y-x|}=0.
$$
$L$ is called the {\em approximate derivative} of $f$ at $x$ and is denoted by  $ap\, Df(x)$.
Recall also that $x\in E\subset\bbbr^k$ is a {\em density point} of $E$ if $\H^k(E\cap B(x,d))/(\omega_kd)^k\to 1$ as $d\to 0$.

If in addition $f:A\to\bbbr$ is Lipschitz, then the
approximate derivative $ap\, Df(x)$ exists for almost every $x \in A$.
This follows from the McShane extension and Rademacher's theorem.
Indeed, if $F:\bbbr^k\to\bbbr$ is a Lipschitz extension of $f$, then $\ap Df(x)$ exists
at {\em all} points of the set 
$$
E=\{x\in A:\, 
\text{$x$ is a density point of $A$ and $F$ is differentiable at $x$}\}. 
$$
Moreover $\ap Df(x)=DF(x)$ at points of the set $E$.

For a Lipschitz map $f=(f_1,f_2,\dots):A \to \ell^{\infty}$, we define the 
{\em component-wise approximate derivative} by 
$$
ap \, Df(x) := 
\left \lceil
\begin{array}{c}
ap \, Df_1(x) \\
ap \, Df_2(x) \\
\vdots
\end{array}
\right \rceil
$$
Since each component $f_i$ is Lipschitz,  $ap\, Df$ exists almost everywhere in $A$.

It is easy to see that the row and column ranks of this $\infty \times k$ matrix are equal, 
and $\rank (ap \, Df(x))$ equals the dimension of the image of $ap \, Df(x)$ in $\ell^{\infty}$.
It follows in particular that $\rank (ap \, Df(x)) \leq k$.

Let $V$ be a linear space of all real sequences. 
In particular, $\ell^\infty\subset V$, 
but we do not equip $V$ with any norm or topology.
If all components of a mapping $g=(g_1,g_2,\ldots):\bbbr^k\to V$ are differentiable at a point $x$, we will say that $g$ is {\em component-wise differentiable} at $x$ and
write
$$
Dg(x) := 
\left \lceil
\begin{array}{c}
Dg_1(x) \\
Dg_2(x) \\
\vdots
\end{array}
\right \rceil
$$
We will also need the following result of Federer (for a proof, see \cite[Theorem~1.69]{MalyZ}, \cite[Theorem~5.3]{simon}, \cite{Whitney}). 
\begin{lemma}
\label{federer}
If $A\subset\bbbr^k$ is measurable and
$f:A\to\bbbr$ is Lipschitz, then for any $\eps>0$ there is a function $g\in C^1(\bbbr^k)$ such that
$$
\H^k(\{x\in A:\, f(x)\neq g(x)\})<\eps.
$$
\end{lemma}
It is easy to see that if $x_0$ is a density point of the set
\begin{equation}
\label{1456}
\{x\in A:\, f(x)=g(x)\},
\end{equation}
then $\ap Df(x_0)$ exists and $\ap Df(x_0)=Dg(x_0)$. In particular $Dg=\ap Df$ almost everywhere in the set \eqref{1456}.

The next lemma was proven in \cite[Proposition 2.3]{HajMal}.
\begin{lemma}
\label{5790}
Let $D \subset \bbbr^{k}$ be a cube or ball, and let $f:D \to \ell^{\infty}$ be $L$-Lipschitz.
Then
$$
\diam(f(D)) \leq C(k)L\mathcal{H}^{k}(D \setminus A)^{1/k},
$$
where $A = \{ x \in D \, : \, Df(x) = 0\}$ and $Df$ is the component-wise derivative of $f$.
\end{lemma}

Finally, in the proof of Corollary~\ref{HSC} we will need
\begin{lemma}
\label{coarea}
If $f:X\to Y$ is a Lipschitz mapping between metric spaces and $A\subset X$, $0\leq m\leq n$, then
$$
\int_Y^*\H^{n-m}(f^{-1}(y)\cap A)\, d\H^m(y)\leq (\lip f)^m
\frac{\omega_{n-m}\omega_m}{\omega_n}\,  \H^n(A).
$$
\end{lemma}

Here $\int^*$ stands for the upper integral and $\lip f$ is a Lipschitz constant of $f$. Federer \cite[2.10.25]{federer}
proved this result under additional assumptions. The general case was obtained by Davies \cite{davies}. A detailed proof is given in \cite[Theorem~2.4]{reichel}.
\begin{corollary}
\label{zero}
If $f:X\to Y$ is  Lipschitz mapping between metric spaces and $A\subset X$, 
$\H^n(A)=0$,
$0\leq m\leq n$, then
$\H^{n-m}(f^{-1}(y)\cap A)=0$ for 
$\H^m$ almost all $y\in Y$.
\end{corollary}

\section{Proof of Theorem \ref{HardSard2}}
\label{SecProof}
The proof is based on techniques developed in \cite{HajMal} (see also \cite{HajMalZim}).
Consider a Lipschitz map $f:A\to\ell^\infty$ defined on a measurable set $A\subset\bbbr^k$.
Our first lemma shows that, 
if the rank of $\ap Df(x)$ is at least $j$ on a set of positive measure,
then, up to local diffeomorphisms, 
$f$ fixes the first $j$ coordinates on some non-null subset.
\begin{lemma}
\label{3.1}
Suppose $f:A\to\ell^\infty$ is a Lipschitz map defined on a measurable set $A\subset\bbbr^k$. If
$\rank (\ap Df(x)) \geq j$ on a subset of $A$ of positive $\H^k$-measure, then there is an open set
$U\subset\bbbr^k$, a set $K\subset A\cap U$ of positive $\H^k$-measure, a bi-Lipschitz
$C^1$-diffeomorphism $G:U\to G(U)\subset\bbbr^k$, 
and a permutation of a finite number of coordinates $\Psi:\ell^\infty\to\ell^\infty$
(which is an isometry of $\ell^\infty$) such that
\begin{equation}
\label{2020}
(\Psi\circ f\circ G^{-1})_i(x)=x_i
\quad
\text{for $i=1,2,\ldots,j$ and $x\in G(K)$.}
\end{equation}
That is for $x\in G(K)$ we have
$$
(\Psi\circ f\circ G^{-1})(x_1,\ldots,x_n)=
(x_1,\ldots,x_j,(\Psi\circ f\circ G^{-1})_{j+1}(x),(\Psi\circ f\circ G^{-1})_{j+2}(x),\ldots).
$$
\end{lemma}
\begin{proof}
By restricting $f$ to the set where $\rank (\ap Df(x)) \geq j$, we may assume that
$\rank (\ap Df(x)) \geq j$ a.e. in $A$.
Since $f=(f_1,f_2,\dots):A \to \ell^{\infty}$ is Lipschitz, 
each component $f_i$ of $f$ is Lipschitz.
Therefore, by applying Lemma~\ref{federer}
component-wise, we may choose $F \subset A$ with $\H^{k}(F)>0$
and a mapping $g=(g_1,g_2,\dots):\mathbb{R}^{k} \to V$ 
with $g_j \in C^1(\bbbr^{k})$ for every $j \in \mathbb{N}$
and such that $g=f$, $Dg = \ap Df$, 
and $\rank Dg = \rank \ap Df \geq j$ on $F$.
Here, as before, $V$ is the vector space consisting of all real valued sequences.
(This is needed since sequences $(g_i(x))_{i=1}^\infty$ are not necessarily bounded.)

\begin{lemma}
Fix $x_0 \in F$.
Under the above assumptions, there is a bi-Lipschitz $C^1$-diffeomorphism 
$G:U \to G(U)\subset\bbbr^k$
defined on a neighborhood $U$ of $x_0$ and a permutation $\Psi:V\to V$ of a finite number of coordinates 
so that
$$
(\Psi\circ g\circ G^{-1})_i(x)=x_i
\quad\text{for $i=1,2,\ldots,j$ and $x\in G(U)$}.
$$
That is, $\Psi\circ g\circ G^{-1}$ fixes the first $j$ coordinates on $G(U)$.
\end{lemma}
\begin{proof}
Since $\rank Dg(x_0)\geq j$, a certain $j\times j$ minor of $Dg(x_0)$ has rank $j$.
By precomposing $g$ with a permutation $\tilde{\Psi}$ of $j$ variables in $\bbbr^k$ and postcomposing it
with a permutation $\Psi$  of $j$ variables in $V$, we have that
$$
\tilde{g}=(\tilde{g}_1,\tilde{g}_2,\ldots)=\Psi\circ g\circ\tilde{\Psi}
$$
satisfies
\begin{equation}
\label{ti382}
\det\left[\frac{\partial\tilde{g}_m}{\partial x_\ell}(\tilde{\Psi}^{-1}(x_0))\right]_{1\leq m,\ell\leq j}\neq 0.
\end{equation}
Let
$$
H(x)=(\tilde{g}_1(x),\ldots,\tilde{g}_j(x), x_{j+1},\ldots,x_k).
$$
It follows from \eqref{ti382} that $\det DH(\tilde{\Psi}^{-1}(x_0))\neq 0$, so $H$ is a diffeomorphism in a neighborhood $\tilde{U}$ of $\tilde{\Psi}^{-1}(x_0)$. Replacing $\tilde{U}$ by a smaller open set, it follows that $H$ is bi-Lipschitz. Now observe that
$$
(\tilde{g}\circ H^{-1})_i(x)=x_i
\quad
\text{for $i=1,2,\ldots,j$ and $x\in H(\tilde{U})$}.
$$
Therefore, if we write $G=H\circ\tilde{\Psi}^{-1}$,
then $\Psi\circ g\circ G^{-1}=\tilde{g}\circ H^{-1}$ 
satisfies the claim of the lemma
on the open set $U=\tilde{\Psi}(\tilde{U})$,
$U$ is a neighborhood of $x_0$, and $G(U)=H(\tilde{U})$.
\end{proof}

Now if $x_0$ is any density point of $F$, 
then the set $K=F\cap U$ has positive measure.
Since $f=g$ on $K$, \eqref{2020} follows because
the permutation of coordinates $\Psi:V\to V$ maps $\ell^\infty\subset V$ to
$\ell^\infty\subset V$ in an isometric way.
This completes the proof of Lemma~\ref{3.1}.
\end{proof}

\begin{lemma}
\label{elllem}
Fix a measurable set $A \subset \bbbr^k$ and $n \leq k$.
Suppose $f:A \to \ell^{\infty}$ is a Lipschitz map. 
If $\Theta^{*n}(f,x)>0$ on a subset of $A$ of positive measure, then
$\rank (\ap Df(x)) \geq n$ on a set of positive measure.
\end{lemma}
\begin{remark}
Note that the above lemmata
involve Lipschitz mappings into $\ell^{\infty}$.
As we will see later, 
this will be sufficient in the setting of any metric space 
since the separable metric space $f(A)$ may be embedded isometrically into $\ell^{\infty}$ via the Kuratowski embedding.
\end{remark}
\begin{remark}
\label{bluerem}
In the following proof, we will see in particular that, for $j\in \{0,1,2,\ldots,n-1\}$, the set of points $x \in A$ where $\Theta^{*n}(f,x)>0$, $\ap Df(x)$ exists, and $\rank (\ap Df(x))=j$ must have measure zero.
\end{remark}
\begin{proof}
Suppose to the contrary that $\Theta^{*n}(f,x)>0$ on a set of positive measure and 
$\rank (\ap Df(x))<n$ almost everywhere in $A$. Then there is $j\in \{0,1,2,\ldots,n-1\}$ and
a set $F\subset A$ with $\H^k(F)>0$ such that $\Theta^{*n}(f,x)>0$ for all $x\in F$,
$\ap Df(x)$ exists and $\rank (\ap Df(x))=j$ for all $x\in F$.

According to Lemma~\ref{3.1}, there is a permutation $\Psi:\ell^\infty\to\ell^\infty$ of a finite
number of variables, an open set $U\subset\bbbr^k$, a set $K\subset F\cap U$ with $\H^k(K)>0$ and a 
bi-Lipschitz
$C^1$-diffeomorphism $G:U\to G(U)\subset\bbbr^k$ such that
$\hat{f}=\Psi\circ f\circ G^{-1}$ defined on $\hat{A}=G(A \cap U)$ satisfies
\begin{equation}
\label{Natalia}
\hat{f}_i(x)=x_i
\quad
\text{for $i=1,2,\ldots,j$ and $x\in \hat{K}$}
\end{equation}
where $\hat{K} = G(K)$.
Note that $\ap D\hat{f}(x)$ exists and
$\rank (\ap D\hat{f})(x)=j$ for all $x\in \hat{K}$, because composition with a diffeomorphism and a permutation $\Psi$ preserve approximate differentiability and the rank of the approximate derivative.

Assume that $x_0$ is a density point of $K$.
Since $\Theta^{*n}(f,x)>0$ for all $x\in K$, in order to arrive to a contradiction, it suffices to show that
$$
\Theta^{*n}(f,x_0)=0.
$$
Note that $y_0=G(x_0)$ is a density point of $\hat{K}=G(K)$ because diffeomorphisms map density points
to density points.

The next lemma shows that it suffices to prove that
\begin{equation}
\label{3.3}
\Theta^{*n}(\hat{f},y_0)= 
\limsup_{d\to 0}
\frac{\H^n_\infty(\hat{f}(B(y_0,d)\cap \hat{A}))}{\omega_n d^n}=0.
\end{equation}
\begin{lemma}
If $\Theta^{*n}(\hat{f},y_0)=0$, then $\Theta^{*n}(f,x_0)=0$.
\end{lemma}
\begin{proof}
Let $d>0$ be so small that $B(y_0,d)\subset G(U)$. Since the diffeomorphism $G^{-1}$ is bi-Lipschitz on $G(U)$, there is a constant $\Lambda>0$ such that
$$
B\left(x_0,\frac{d}{\Lambda}\right)\subset G^{-1}(B(y_0,d)).
$$
Since the permutation of coordinates $\Psi:\ell^\infty\to\ell^\infty$ is an isometry, it follows that
$\H^n_\infty(\hat{f}(E))=\H^n_\infty(f(G^{-1}(E)))$ for any set $E$ in the domain of $\hat{f}$. Therefore
$$
\H^n_\infty(\hat{f}(B(y_0,d)\cap \hat{A}))=
\H^n_\infty(f(G^{-1}(B(y_0,d)\cap \hat{A}))\geq
\H^n_\infty(f(B(x_0,d/\Lambda)\cap A)),
$$
so
\begin{equation*}
\begin{split}
\Theta^{*n}(\hat{f},y_0)
&=
\limsup_{d\to 0}
\frac{\H^n_\infty(\hat{f}(B(y_0,d)\cap \hat{A}))}{\omega_n d^n}\\
&\geq
\Lambda^{-n}\limsup_{d\to 0}\frac{\H^n_\infty(f(B(x_0,d/\Lambda)\cap A))}{\omega_n (d/\Lambda)^n}
=
\Lambda^{-n}\Theta^{*n}(f,x_0)
\end{split}    
\end{equation*}
and the lemma follows.
\end{proof}
To conclude the proof of \eqref{3.3}, we will apply the following lemma.
\begin{lemma}
\label{HaMaLemma}
Assume $d > 0$ is such that $Q(y_0,d) \subset G(U)$ and
$$
\H^{k}(Q(y_0,d) \setminus \hat{K}) < \left( \frac{d}{M} \right)^{k}
$$
for some positive integer $M$.
Then $\hat{f}(Q(y_0,d) \cap \hat{A})$ can be covered by $M^j$ balls
of radius $CLdM^{-1}$ for some constant $C = C(k,n) > 0$, where $L$ is the Lipschitz constant of $\hat{f}$.
In particular, we have
$$
\H_{\infty}^n(\hat{f}(Q(y_0,d) \cap \hat{A})) \leq \omega_n(CLd)^n M^{j-n}.
$$
\end{lemma}
Before proving this lemma, 
we will see how it can be used to prove 
\eqref{3.3}.
Let $\eps > 0$. Fix a positive integer $M$
such that $(CL)^n M^{j-n} < \eps$.
(This is possible since $j-n < 0$.)
Since $y_0$ is a density point of $\hat{K}$, 
there is $\delta>0$ such that for $0<d<\delta$,
$Q(y_0,d) \subset G(U)$ 
satisfies
$$
\H^{k}(Q(y_0,d) \setminus \hat{K}) < \frac{\H^{k}(Q(y_0,d))}{M^k} = \left( \frac{d}{M} \right)^{k}.
$$
Hence, by Lemma~\ref{HaMaLemma}, we have
$$
\frac{\H_{\infty}^n(\hat{f}(Q(y_0,d) \cap \hat{A}))}{\omega_n d^n} \leq (CL)^n M^{j-n} < \eps
\quad
\text{for $0<d<\delta$}
$$
which, along with \eqref{1410}, implies that $\Theta^{*n}(\hat{f},y_0)=0$.
That completes the proof of \eqref{3.3} once Lemma~\ref{HaMaLemma} has been verified.
The proof of Lemma~\ref{HaMaLemma} is nearly identical to the proof of \cite[Lemma~2.7]{HajMal},
but we will include it here for completeness.

\begin{proof}[Proof of Lemma~\ref{HaMaLemma}]
Assume that a positive integer $M>0$ and $d > 0$ satisfy $Q(y_0,d) \subset G(U)$ and
$$
\H^{k}(Q(y_0,d) \setminus \hat{K}) < \left( \frac{d}{M} \right)^{k}.
$$
Since the result is translation invariant, we may assume without loss of generality that 
$$
Q(y_0,d)=Q = [0,d]^j \times [0,d]^{k-j}.
$$
According to Lemma~\ref{2.7}, the $L$-Lipschitz mapping $\hat{f}:Q \cap \hat{A} \to \ell^{\infty}$ 
admits an $L$-Lipschitz extension $\tilde{f}:Q \to \ell^{\infty}$.
According to Rademacher's theorem, $\tilde{f}$ is component-wise differentiable for almost all points in $Q$.

Divide $[0,d]^j$ into $M^j$ cubes $\{Q_{\nu}\}_{\nu = 1}^{M^j}$ with pairwise disjoint interiors each of edge length $d/M$.
It suffices to show that each set
$$
\hat{f}((Q_{\nu} \times [0,d]^{k-j}) \cap \hat{A}) 
\subset 
\tilde{f}(Q_{\nu} \times [0,d]^{k-j})
$$
is contained in an $\ell^{\infty}$-ball
of radius $CLdM^{-1}$ for some constant $C=C(k,n)>0$.
By our assumptions, for each $\nu$ we have
$$
\mathcal{H}^k((Q_{\nu} \times [0,d]^{k-j}) \setminus \hat{K})
\leq 
\mathcal{H}^k(Q \setminus \hat{K})
<
\left( \frac{d}{M} \right)^{k}.
$$
Hence
$$
\mathcal{H}^k((Q_{\nu} \times [0,d]^{k-j}) \cap \hat{K}) > \left( M^{-j} - M^{-k} \right)d^k.
$$
According to Fubini's Theorem,
we may therefore choose some $\rho \in Q_{\nu}$ such that
$$
\mathcal{H}^{k-j}((\{\rho\} \times [0,d]^{k-j}) \cap \hat{K}) > \left( 1 - M^{j-k} \right)d^{k-j}
$$
and $\tilde{f}$ is component-wise
differentiable at almost all points of $\{\rho\} \times [0,d]^{k-j}$. 
Hence
\begin{equation}
\label{2.1}
\mathcal{H}^{k-j}((\{\rho\} \times [0,d]^{k-j}) \setminus \hat{K}) < \left( \frac{d}{M} \right)^{k-j}.
\end{equation}
According to \eqref{Natalia}, $\hat{f}$ fixes the first $j$ coordinates in $\hat{K}$. 
Since $\hat{f}=\tilde{f}$ in $\hat{K}$ and
$\rank (\ap D\hat{f}(x))=j$ everywhere in  $\hat{K}$, it follows that 
$\tilde{f}_i(x)=x_i$ for $i=1,2,\ldots,j$ and $x\in \hat{K}$ and 
$\rank D\tilde{f}(x)=j$ almost everywhere in $\hat{K}$.
Therefore, the component-wise
derivative of $\tilde{f}$ along $\{\rho\}\times [0,d]^{k-j}$ vanishes at almost all points in 
$(\{\rho\}\times [0,d]^{k-j})\cap \hat{K}$. 
That is
$$
D\big(\tilde{f}\big|_{\{\rho\}\times [0,d]^{k-j}}\big)=0
\quad
\text{a.e. in $(\{\rho\}\times [0,d]^{k-j})\cap \hat{K}$.}
$$
Therefore Lemma~\ref{5790} applied to $\tilde{f}:\{\rho\} \times [0,d]^{k-j}\to\ell^\infty$ (with $k$ replaced by $k-j$)
together with \eqref{2.1} yield
$$
\diam(\tilde{f}(\{\rho\} \times [0,d]^{k-j}))
\leq
CL\mathcal{H}^{k-j}((\{\rho\} \times [0,d]^{k-j}) \setminus \hat{K})^{1/(k-j)}
\leq 
CLdM^{-1}.
$$
Since the distance from any point in $Q_{\nu} \times [0,d]^{k-j}$ to the set
$\{\rho\} \times [0,d]^{k-j}$ is at most $\diam(Q_\nu ) = \sqrt{j} dM^{-1}$
and $\tilde{f}$ is $L$-Lipschitz,
this implies that 
$$
\diam(\tilde{f}(Q_\nu \times [0,d]^{k-j})) 
\leq 
CLdM^{-1}
$$
(for a larger value of $C$).
This proves Lemma~\ref{HaMaLemma}.
\end{proof}
This also completes the proof of \eqref{3.3} and hence that of Lemma~\ref{elllem}.
\end{proof}

We now can finish the proof of Theorem~\ref{HardSard2}.

\begin{proof}[Proof of Theorem~\ref{HardSard2}]

Since $f(A)\subset X$ is a separable metric space, there is an isometric embedding
$\kappa:f(A)\to\ell^\infty$ (see Lemma~\ref{kura}).
The mapping $\kappa$ is $1$-Lipschitz. According to Lemma~\ref{2.7}, the map $\kappa$ admits a $1$-Lipschitz extension $\K:X\to\ell^\infty$.

Then $\bar{f}=\K\circ f=\kappa\circ f:A\to\ell^\infty$ is Lipchitz and
$\Theta^{*n}(\bar{f},x) > 0$ on a subset of $A$ with positive measure
(composition with an isometric map does not change the upper density).

It follows from Lemma~\ref{elllem} (with $k=n+m$) that $\rank \ap D\bar{f} \geq n$ on a set of positive measure. Therefore, according to Lemma~\ref{3.1}, there is an open set $U\subset\bbbr^{n+m}$,
a subset $K\subset A\cap U$ with $\H^{n+m}(K)>0$,
a bi-Lipschitz $C^1$-diffeomorphism $G:U\to G(U)\subset\bbbr^{n+m}$ and a permutation of finitely many coordinates
$\Psi:\ell^\infty\to\ell^\infty$ such that
\begin{equation}
\label{Michal}
(\Psi\circ \bar{f}\circ G^{-1})_i(x)=x_i
\quad\text{for $i=1,2,\ldots,n$ and $x\in G(K)$.}
\end{equation}
Let
$$
P:\ell^\infty\to\bbbr^n,
\quad
P(x_1,x_2,\ldots)=(x_1,x_2,\ldots,x_n)
$$
be the projection onto the first $n$ coordinates. Then $P$ is $1$-Lipschitz as a mapping to $\bbbr^n$ equipped with the $\ell^\infty_n$ norm, $\Vert (x_1,\ldots,x_n)\Vert_\infty=\max_i|x_i|$ and $\sqrt{n}$-Lipschitz as a mapping to $\bbbr^n$ with the Euclidean metric.
Therefore, it follows that the mapping 
$$
\pi:X\to\bbbr^n,
\quad
\pi=P\circ\Psi\circ\K
$$
is  $1$-Lipschitz as a mapping to $\bbbr^n$ equipped with the norm $\ell^\infty_n$ and $\sqrt{n}$-Lipschitz as a mapping to $\bbbr^n$ with the Euclidean metric (see Remark~\ref{R145}).

If we swith to notation
$$
(x,y)=(x_1,\ldots,x_n,y_1,\ldots,y_m):=
(x_1,\ldots,x_n,x_{n+1},\ldots,x_{n+m}),
$$
then clearly, \eqref{Michal} means that $(\pi\circ f\circ G^{-1})(x,y)=x$ for $(x,y)\in G(K)$ which 
completes the proof of the statement (B).

To prove (A), suppose to the contrary that $\H^n(f(A))=0$. Then $\H^n(f(K))=0$ and hence
\begin{equation}
\label{Joasia}   
\H^n((\pi\circ f\circ G^{-1})(G(K))=
\H^n(\pi(f(K)))\leq(\sqrt{n})^n\H^n(f(K))=0,
\end{equation}
because the $\sqrt{n}$-Lipschitz map $\pi$ can increase the $\H^n$-measure no more
than by a factor $(\sqrt{n})^n$.
On the other hand, $G(K)$ has positive $\H^{n+m}$-measure so it follows from Fubini's theorem that
its projection $(\pi\circ f\circ G^{-1})(G(K))$ onto the first $n$-coordinates has positive 
$\H^n$-measure which contradicts \eqref{Joasia}.

Parts (C) and (D) are easy consequences of part (B) as follows.
Write $F = f \circ G^{-1}$.

Let $(x',y')\in F^{-1}(F(x,y))\cap G(K)$. Then $F(x',y')=F(x,y)$ so
$x'=\pi(F(x',y'))=\pi(F(x,y))=x$ and hence $(x',y')=(x,y')\in \{ x\}\times\bbbr^m$ which proves (C).

To prove (D), fix $y\in\bbbr^m$ and let $(x_1,y),(x_2,y)\in G(K)$.
Let $\Lambda$ be the Lipschitz constant of $F$ on $G(K)$. Since
$\pi:X\to(\bbbr^n,\ell^\infty_n)$ is $1$-Lipschitz we have
\begin{equation*}
\begin{split}
n^{-1/2}|x_1-x_2|&\leq \Vert x_1-x_2\Vert_\infty=
\Vert\pi(F(x_1,y))-\pi(F(x_2,y))\Vert_\infty \\
&\leq d(F(x_1,y),F(x_2,y))\leq\Lambda|x_1-x_2|        
\end{split}
\end{equation*}
which proves (D) along with the estimate \eqref{404}.
The proof is complete.
\end{proof}

\section{Proof of Corollary~\ref{HSC}}
\label{HSCProof}
Since $\Theta^{*n}(f,x)>0$ almost everywhere in $A$, we can exhaust $A$ up to a set of $\H^{n+m}$ measure zero by a countable family of pairwise disjoint sets of positive $\H^{n+m}$ measure $\{K_i\}$, where each of the sets $K=K_i$ satisfies claim (B) of Theorem~\ref{HardSard2}.
Say $\{ G_i \}$ are the associated bi-Lipschitz $C^1$-diffeomorphisms.

Let $W=\bigcup_{i=1}^\infty K_i$ and $Z=A\setminus W$ so $\H^{n+m}(Z)=0$.
Let $f_i=f|_{K_i}$ and let $F_i=f_i\circ G_i^{-1}$. Since $F_i$ is defined on $G_i(K_i)$ only, we have from part (C) of Theorem~\ref{HardSard2} that for {\em any} $z\in X$, $F_i^{-1}(z)$ is contained in an $m$-dimensional affine subspace of $\bbbr^{n+m}$ and hence $f_i^{-1}(z)=G_i^{-1}(F_i^{-1}(z))$ is contained in an $m$-dimensional submanifold (of class $C^1$). Therefore, for any $z\in X$,
$$
f^{-1}(z)\cap W=\bigcup_{i=1}^\infty f_i^{-1}(z)
$$
is countably $\H^m$-rectifiable as it is contained in a countable union of $m$-manifolds, and it remains to observe from Corollary~\ref{zero} that $\H^m (f^{-1}(z)\setminus W) = \H^m (f^{-1}(z)\cap Z)=0$
for $\H^n$ almost all $z\in X$.
\hfill $\Box$

\section{Comparing Theorems \ref{HardSard2} and \ref{AS}}
\label{SecComp}

Recall the $(n,m)$-Hausdorff content which was defined in \eqref{AS12}.
As mentioned in the introduction,
the assumption
that $\Theta^{*n}(f,x)>0$ on a set of positive measure
in Theorem~\ref{HardSard2}
is weaker than the assumption of positive $(n,m)$-Hausdorff content of a cube in Theorem~\ref{AS}.
We see this fact in the following proposition,
the proof of which follows easily from the Vitali Covering Theorem.

\begin{proposition}
\label{elllem2}
Suppose $Q \subset \mathbb{R}^{n+m}$ is a cube,
$X$ is a metric space,
and $f:Q \to X$ is Lipschitz.
Then
$$
\mathcal{H}_{\infty}^{n,m}(f,Q) \leq 
\frac{\omega_n}{2^n}(n+m)^{n/2}\int_Q \Theta_*^n(f,x) \, dx
\leq \frac{\omega_n}{2^n}(n+m)^{n/2}\int_Q \Theta^{*n}(f,x) \, dx.
$$
\end{proposition}

\begin{proof}
In this proof $Q(x,d)$ and $\overbar{Q}(x,d)$ will denote open and closed cubes in $\bbbr^{n+m}$ respectively.
Note that $Q(x,d)\subset B(x,\lambda d)$, where $\lambda=\frac{\sqrt{n+m}}{2}$.

The function $\Theta_*^n(f,\cdot)$ is integrable on $Q$ since it is bounded.
Fix $\eps > 0$.
Denote by $A$ the set of all points in the interior of $Q$ 
which are Lebesgue points of the function $\Theta_*^{n}(f,\cdot)$.
Fix a point $x \in A$,
and choose $d_x>0$ small enough so that $Q(x,d_x)\subset B(x,\lambda d_x) \subset Q$.
Choose a sequence $\{ d_x^i \}_{i=1}^{\infty}$ with $d_x > d_x^i \searrow 0$
satisfying the following for each $d = d_x^i$:
$$
\Theta_*^{n}(f,x) \leq \frac{1}{|Q(x,d)|} \int_{Q(x,d)} \Theta_*^{n}(f,y) \, dy + 
\frac{\eps}{2}
$$
and
$$
\frac{\mathcal{H}_{\infty}^n(f(Q(x,d)))}{\omega_n (\lambda d)^n} \leq 
\frac{\mathcal{H}_{\infty}^n(f(B(x,\lambda d)))}{\omega_n (\lambda d)^n}\leq
\Theta_*^n(f,x) + \frac{\eps}{2}.
$$
Both inequalities imply that
$$
\H^n_\infty(f(Q(x,d)))d^m\leq
\omega_n\lambda^n\left(\int_{Q(x,d)}\Theta_*^n(f,y)\, dy+\eps d^{n+m}\right)
\quad
\text{for all $x\in A$ and all $d=d_x^i$.}
$$
The collection of closed cubes  
$$
\mathcal{Q}=\{ \overbar{Q}(x,d_x^i):\, x\in A,\ i\in\bbbn\}
$$ 
is a fine Vitali covering of $A$.
Thus there is a countable, pairwise disjoint collection of cubes $\{ \overbar{Q}(x_j,d_j) \}$ in $\mathcal{Q}$
so that
$$
\H^{n+m}\left(Q\setminus\bigcup_{j}Q(x_j,d_j)\right)=
\H^{n+m}\left(A \setminus \bigcup_j \overbar{Q}(x_j,d_j) \right) = 0.
$$
Since the cubes $Q(x_j,d_j)$ are open, pairwise disjoint, contained in $Q$, and they cover $Q$ up to a set of measure zero, the definition of $\H^{n,m}_\infty(f,Q)$ yields
\begin{align*}
\H_{\infty}^{n,m}(f,Q) 
\leq \sum_j \H^n_{\infty}(f(Q(x_j,d_j)))d_j^m 
&\leq \omega_n\lambda^n\left(\sum_j\int_{Q(x_j,d_j)}\Theta_*^n(f,y)\, dy+\eps\sum_j d_j^{n+m}\right)\\
&= \omega_n\lambda^n\left(\int_Q \Theta_*^n(f,y) \, dy + \eps |Q|\right).
\end{align*}
Sending $\eps \to 0$ gives the desired result.
\end{proof}

The next result shows that $\Theta^n_*(f,x)$ is in fact equal to the Jacobian of $f$ when $f$ is a Lipschitz mapping to $\bbbr^n$. This result is related to Lemma~\ref{elllem}. 
Consider a mapping $f:\mathbb{R}^{n+m} \to \mathbb{R}^n$ which is differentiable at $x \in \bbbr^{n+m}$.
Define the Jacobian $|J^nf|(x)$ at $x$ as follows:
$$
|J^nf|(x) = \sqrt{\det (Df)(Df)^T(x)}.
$$
Geometrically, it follows that, 
when $\rank Df(x)=n$,
the Jacobian satisfies
\begin{equation}
\label{wrx}
|J^nf|(x) = \frac{\mathcal{H}^n (W_{x,r})}{\omega_n r^n}
\quad
\text{for any $r>0$,}
\end{equation}
where 
\begin{equation}
\label{wrx2}
W_{x,r} = f(x) + Df(x) (B(0,r))
\quad 
\text{for } 
B(0,r)\subset T_x\bbbr^{n+m}
\end{equation}
is the ellipsoid approximation (in $\bbbr^n$) of $f(B(x,r))$.
This Jacobian plays an important role in the so called co-area formula \cite[Theorem~2.7.3]{ziemer}.

Observe that if $\pi:T_x\bbbr^{n+m}\to (\ker Df(x))^\perp\subset T_x\bbbr^{n+m}$ is the orthogonal projection
onto the $n$-dimensional subspace $(\ker Df(x))^\perp$, then 
$W_{x,r}=f(x)+Df(x)(\pi(B(0,r))$, so $W_{x,r}$ is (up to a translation by the vector $f(x)$) the image of the $n$-dimensional ball $\pi(B(0,r))\subset (\ker Df(x))^\perp$ of radius $r$ under the linear map $Df(x)$. That is, $|J^nf|(x)$ is the ratio of the volume of the ellipsoid $W_{x,r}$ to the volume of $\pi(B(0,r))$.

If the rank of $Df(x)$ is less than $n$, we have $|J^nf|(x)=0$. Therefore
$|J^n f|(x)>0$ if and only if $\rank Df(x)=n$.
We similarly define the Jacobian of any Lipschitz mapping $f:\bbbr^{n+m} \supset A \to\bbbr^n$ using the approximate derivative.

\begin{proposition}
\label{euclem}
Let $f:A\to\bbbr^n$ be a Lipschitz map defined on a measurable set $A\subset\bbbr^{n+m}$. Then
\begin{equation}
\label{TTJ}
\Theta_*^n(f,x)=\Theta^{*n}(f,x) =|J^nf|(x)
\end{equation}
for almost every $x \in A$.
\end{proposition}
Note that
combining this result with Proposition~\ref{elllem2} gives the following
for any cube $Q \subset \mathbb{R}^{n+m}$ 
and any Lipschitz
$f:Q\to\bbbr^n$: 
\begin{equation}
\label{asasas}
\mathcal{H}_{\infty}^{n,m}(f,Q) \leq \frac{\omega_n}{2^n}(n+m)^{n/2} 
\int_{Q} |J^nf|(x) \, dx.
\end{equation}
This inequality is essentially Lemma 6.13 in \cite{AzzSch}.
\begin{proof}
Assume first that $f:\bbbr^{n+m}\to\bbbr^n$ is an $L$-Lipschitz mapping defined on all of $\bbbr^{n+m}$. It suffices to prove that \eqref{TTJ} holds true at all points of differentiability of $f$.

Let $x\in\bbbr^{n+m}$ be a point of differentiability of $f$.
Given $L>\eps>0$, there is $\delta>0$ such that
\begin{equation}
\label{eq4.3}
|f(y)-f(x)-Df(x)(y-x)|<\eps r
\quad
\text{for all $0<r<\delta$ and $y\in B(x,r)$.}
\end{equation}
Assume first that $|J^nf|(x)=0$. We will show that $\Theta_*^n(f,x)=\Theta^{*n}(f,x) =0$.

Let $W_x=f(x)+Df(x)(T_x\bbbr^{n+m})$ be an affine space through $f(x)$ (which is the image of the derivative in $\bbbr^n$).
Since $|J^nf|(x)=0$, we have that $\dim W_x\leq n-1$ and hence
\begin{equation}
\label{7654}
f(B(x,r))\subset B(f(x),Lr)\cap \{z\in\bbbr^n:\,\dist (z,W_x)<\eps r\}
\quad
\text{for $0<r<\delta$.}
\end{equation}
Since $\dim W_x=k\leq n-1$ we have that
$$
\H^n_\infty(f(B(x,r)))\leq C(n)\eps L^{n-1}r^{n}.
$$
Indeed, the $k$-dimensional affine ball $B(f(x),Lr)\cap W_x\subset\bbbr^n$ can be covered by
$$
C\left(\frac{Lr}{\eps r}\right)^k\leq C\left(\frac{L}{\eps}\right)^{n-1}
$$
balls in $\bbbr^n$ of radius $\eps r$ and centered at the points of 
$B(f(x),Lr)\cap W_x$. Then the balls with radii $2\eps r$ and the same centers cover the set on the right hand side of \eqref{7654}, and hence they also cover $f(B(x,r))$. Since a ball of radius $2\eps r$ has diameter $4\eps r$ we have that
$$
\H^n_\infty(f(B(x,r)))\leq \frac{\omega_n}{2^n}(4\eps r)^n C\left(\frac{L}{\eps}\right)^{n-1}
=
C(n) \omega_n \eps r^n L^{n-1}.
$$
Therefore,
$$
\frac{\H^n_\infty(f(B(x,r)))}{\omega_n r^n}\leq C\eps L^{n-1}
\quad\text{for $0<r<\delta$}
$$
which readily yields  $\Theta_*^n(f,x)=\Theta^{*n}(f,x) =0$.

Assume now that $|J^nf|(x)>0$. Let $W_{x,r}=f(x)+Df(x)(B(0,r))$ be the ellipsoid considered in \eqref{wrx2}. 
Let $0<\lambda_1\leq\lambda_2\leq\ldots\leq\lambda_n$ be the singular values of $Df(x)$ i.e.,
the lengths of the semiaxes of $W_{x,r}$ are $0<\lambda_1 r\leq\lambda_2 r\leq\ldots\leq\lambda_n r$.
($\lambda_1>0$ because $|J^nf|(x)>0$).

Consider the three concentric and homothetic ellipsoids
(we further assume $0<\eps<\lambda_1$ so $1-\eps/\lambda_1>0$)
$$
W_{x,(1-\eps/\lambda_1)r}\subset W_{x,r}\subset W_{x,(1+\eps/\lambda_1)r}.
$$
The distance between the boundary of the ellipsoid $W_{x,r}$ and the boundaries of each of the other two ellipsoids equals $\eps r$ since the distance between the homothetic ellipsoids is measured along the shortest semiaxes (as an easy exercise for the Lagrange multipliers). Therefore it follows from \eqref{eq4.3} that
\begin{equation}
\label{rudy102}
W_{x,(1-\eps/\lambda_1)r}\subset f(B(x,r))\subset W_{x,(1+\eps/\lambda_1)r}
\quad
\text{for $0<r<\delta$.}
\end{equation}
Indeed, the right inclusion follows immediately from \eqref{eq4.3}. The proof of the left inclusion is more intricate. Suppose to the contrary that  
$$
z\in W_{x,(1-\eps/\lambda_1)r}\setminus f(B(x,r)).
$$
Then using a `radial' projection from $z$ and estimate \eqref{eq4.3} one can construct a retraction of the ellipsoid $W_{x,r}$ to its boundary which is a contradiction. We leave details of a construction of a retraction to the reader.

It follows from Lemma~\ref{h=h} and \eqref{wrx} that for any $R>0$
$$
\H^n_\infty(W_{x,R})=\H^n(W_{x,R})=|J^n f|(x)\omega_n R^n
$$
so \eqref{rudy102} implies that for $0<r<\delta$ we have
$$
|J^nf|(x)\left(1-\frac{\eps}{\lambda_1}\right)^n\leq
\frac{\H^n_\infty(f(B(x,r)))}{\omega_n r^n}\leq
|J^nf|(x)\left(1+\frac{\eps}{\lambda_1}\right)^n
$$
and letting $\eps\to 0$ yields \eqref{TTJ}.

Note that the proof presented above is enough to establish \eqref{asasas}.

We can now proceed to the proof of the result in the general case when $f:\bbbr^{n+m} \supset A \to\bbbr^n$ is Lipschitz.

Let $\tilde{f}:\bbbr^{n+m}\to\bbbr^n$ be a Lipschitz extension of $f$. 
Assume that $L$ is the Lipschitz constant of $\tilde{f}$.
Note that
$|J^n f|=|J^n \tilde{f}|$ at almost all points of $A$, and, by the proof presented above, $|J^n\tilde{f}|(x)=\Theta_*^n(\tilde{f},x)=\Theta^{*n}(\tilde{f},x)$ for almost all
$x\in\bbbr^{n+m}$. Note also that $\Theta^{*n}(\tilde{f},x)\geq \Theta^{*n}({f},x)$, because
in the case of $\Theta^{*n}(\tilde{f},x)$ we consider the Hausdorff content of $\tilde{f}(B(x,r))$ while in the case of $\Theta^{*n}({f},x)$ we only consider the Hausdorff content of $f(B(x,r)\cap A)=\tilde{f}(B(x,r)\cap A)$.

Since for almost all $x\in A$ we have
$$
|J^n f|(x)=|J^n\tilde{f}|(x)=\Theta_*^n(\tilde{f},x)=
\Theta^{*n}(\tilde{f},x)\geq \Theta^{*n}(f,x)\geq \Theta_*^n(f,x),
$$
it suffices to show that
\begin{equation}
\label{L10}
\Theta_*^n(f,x)\geq |J^n\tilde{f}|(x)
\quad
\text{for almost all $x\in A$.}
\end{equation}
For almost all $x\in A$ such that $|J^nf|(x)=0$,
this is particularly easy. Indeed, we have
$$
\Theta_*^n(f,x)\geq 0=|J^n f|(x)=|J^n\tilde{f}|(x),
$$
so \eqref{L10} is obvious.

We are left with the case when $|J^nf|(x)>0$.
Since we want to prove \eqref{L10} almost everywhere, we can assume that 
$x$ is a density point of $A$ and $\tilde{f}$ is differentiable at $x$. 
Then $|J^n\tilde{f}|(x)=\Theta^{*n}(\tilde{f},x)=\Theta_*^n(\tilde{f},x)$, 
$\ap Df(x)=D\tilde{f}(x)$, 
and $|J^n\tilde{f}|(x)=|J^nf|(x)>0$.
In particular, we have $\rank D\tilde{f}(x) = n$.

The idea of the rest of the proof is simple. 
Since $x$ is a density point of $A$, for small $r>0$, the content
$\H^n_\infty(f(B(x,r)\cap A))=\H^n(\tilde{f}(B(x,r)\cap A))$
{\bf is not much smaller} than  $\H^n(\tilde{f}(B(x,r)))=\H^n_\infty(\tilde{f}(B(x,r)))$.
Therefore dividing by $\omega_nr^n$ and passing to the liminf as $r\to 0$ gives
$$
\liminf_{r\to 0}\frac{\H^n_\infty(f(B(x,r)\cap A))}{\omega_n r^n}
+\eps
\geq
\liminf_{r\to 0}\frac{\H^n_\infty(\tilde{f}(B(x,r)))}{\omega_n r^n}=\Theta_*^n(\tilde{f},x)=
|J^n\tilde{f}|(x)
$$
for all $\eps > 0$.
Thus the main focus in the argument presented below is proving the phrase
``{\bf is not much smaller}''.
While the idea of the proof presented below is very geometric and relatively simple, the details are not.

By translating the coordinate system we may assume that $x=0$.
The ellipsoid $W_{0,r}=\tilde{f}(0)+D\tilde{f}(0)(B(0,r))$ is the image of the ball 
$B^{n+m}(0,r)\subset T_0\bbbr^{n+m}$.
By abusing notation we will identify the tangent space $T_0\bbbr^{n+m}$ with $\bbbr^{n+m}$. 
For example the same notation will be used for the ball $B^{n+m}(0,r)$ in the tangent space $T_0\bbbr^{n+m}$,
and for the ball $B^{n+m}(0,r)=0+B^{n+m}(0,r)$ in $\bbbr^{n+m}$.

Since $\rank D\tilde{f}(0)=n$, we have $\dim\ker D\tilde{f}(0)=m$.
Rotating the coordinate system in $\bbbr^{n+m}$ we may assume that
$$
\bbbr^{n+m}=T_0\bbbr^{n+m}=(\ker D\tilde{f}(0))^\perp\oplus (\ker D\tilde{f}(0))=\bbbr^n\oplus \bbbr^m.
$$
Let 
$$
\pi:\bbbr^n\oplus\bbbr^m\to\bbbr^n\oplus\{ 0\}\subset\bbbr^n\oplus\bbbr^m
$$
be the orthogonal projection.
Note that the $n$-dimensional ball in the tangent space
$$
B^n_0(r):=\pi(B^{n+m}(0,r))
=(\bbbr^n\times\{ 0\})\cap B^{n+m}(0,r)\subset T_0\bbbr^{n+m}
$$
has radius $r$ and
$$
W_{0,r}=\tilde{f}(0)+D\tilde{f}(0)(B^n_0(r)).
$$
Let $\eps>0$ be given, then there is a positive integer $M$ such that
\begin{equation}
\label{1492}
|J^n\tilde{f}|(0)\left(1-\frac{\sqrt{2}}{\lambda_1 M}\right)^n\left(1-\frac{1}{2M}\right)- \frac{L^n}{2M}
\geq
|J^n\tilde{f}|(0) - \eps,
\end{equation}
where $0<\lambda_1\leq\lambda_2\leq\ldots\leq\lambda_n$ are the singular values of $D\tilde{f}(0)$.

For any $r>0$ and any $0<t<1$ let
$$
V_{t,r}=(\bbbr^n\times B^m(0,tr))\cap B^{n+m}(0,r)
$$
be the $tr$-cylinder around $B^n_0(r)$ inside of the ball $B^{n+m}(0,r)$.
Clearly $\H^{n+m}(V_{t,r})<\omega_n r^n\cdot\omega_m(tr)^m$ 
because $V_{t,r}\subset B^n(0,r)\times B^m(0,tr)$.
Also, when $t$ is small,
the volume of $V_{t,r}$ must be close to the volume of this product of balls in the following sense:
$$
\lim_{t\to 0}\frac{\H^{n+m}(V_{t,r})}{\omega_n r^n\cdot\omega_m (tr)^m}=1.
$$
Thus there is $0<t_M<(1-\frac{1}{2M})^{1/n}$
such that
\begin{equation}
\label{L1}
\left(1-\frac{1}{4M}\right)\omega_n\omega_mr^{n+m}t_M^m<
\H^{n+m}(V_{t_M,r})<
\omega_n\omega_m r^{n+m} t_M^m.
\end{equation}
Note that $t_M$ depends on $M$ but not on $r$ because 
$V_{t,r} = rV_{t,1}$ (where $rE := \{rx \, : \, x \in E\}$ for $E \subset \bbbr^{n+m}$).

Since $0\in A$ is a density point of $A$, we may choose $\delta>0$ depending on $M$ so that for $0<r<\delta$ we have
\begin{equation}
\label{L4}
\H^{n+m}(V_{t_M,r}\setminus A)\leq
\H^{n+m}(B^{n+m}(0,r)\setminus A)<\frac{1}{4M}\omega_n\omega_m r^{n+m} t_M^m
\end{equation}
and hence
\begin{equation}
\label{L2}
\H^{n+m}(V_{t_M,r}\cap A)=
\H^{n+m}(V_{t_M,r})-\H^{n+m}(V_{t_M,r}\setminus A)>
\left(1-\frac{1}{2M}\right)\omega_n\omega_m r^{n+m} t_M^m.
\end{equation}
Since $\tilde{f}$ is differentiable at $0$, we may also assume 
(by taking, if necessary, a smaller $\delta>0$ depending on $M$) that
\begin{equation}
\label{L5}
|\tilde{f}(y)-\tilde{f}(0)-D\tilde{f}(0)y|<
\frac{r}{M}
\quad
\text{for all $0<r<\delta$ and $y\in B^{n+m}(0,r)$}.
\end{equation}
Let $0<r<\delta$.
For $b\in\bbbr^m$ we define
$$
B^n_b(r)=(\bbbr^n\times\{ b\})\cap B^{n+m}(0,r).
$$
If we regard $V_{t_M,r}$ as a cylinder with base $B^m(0,t_M r)$ (and with spherical caps),
then the fibers (orthogonal to the base) are the $n$-balls $B^n_b(r)$ where $b$ ranges over $B^m(0,t_M r)$.

We claim that the set of $b\in B^m(0,t_M r)$ which satisfy
\begin{equation}
\label{L3}
\H^n(B^n_b(r)\cap A)>\left(1-\frac{1}{2M}\right)\, \omega_n r^n
\end{equation}
has positive $\H^m$-measure.
Indeed, suppose to the contrary that
$$
\H^n(B^n_b(r)\cap A)\leq\left(1-\frac{1}{2M}\right)\, \omega_n r^n
\quad
\text{for } \H^m \text{-almost all $b\in B^m(0,t_M r)$}.
$$
Then it follows from Fubini's theorem that 
$$
\H^{n+m}(V_{t_M,r}\cap A)\leq\left(1-\frac{1}{2M}\right)\, \omega_nr^n\cdot \omega_m (t_M r)^m
$$
which contradicts \eqref{L2}.
In other words, we have shown that the set of fibers of $V_{t_M,r}$
which see a ``large'' part of $A$ 
has positive $\H^m$-measure.

Let $b\in B^m(0, t_M r)$ be such that \eqref{L3} is satisfied. Then the radius $R$ of the ball $B^n_b(r)$ satisfies
$$
r\geq R>\left(1-\frac{1}{2M}\right)^{1/n}r,
$$
and since $D \tilde{f}(0)$ vanishes in 
the direction of $(0,b)$
$$
\tilde{f}(0)+D\tilde{f}(0)(B^n_b(r))=
\tilde{f}(0)+D\tilde{f}(0)(\pi(B^n_b(r)))=
\tilde{f}(0)+D\tilde{f}(0)(B^n_0(R))=W_{0,R},
$$
$$
\H^n(W_{0,R})=|J^n\tilde{f}|(0)\omega_n R^n.
$$
Recall that
$$
0<t_M<\left(1-\frac{1}{2M}\right)^{1/n}
\quad
\text{and}
\quad
b\in B^m(0,t_Mr).
$$
Therefore, 
$|b| < t_M r < R$.
Thus by \eqref{L5} and the Pythagorean theorem, we have
\begin{equation}
\label{L100}
|\tilde{f}(y)-(\tilde{f}(0)+D\tilde{f}(0)y)|
\leq
M^{-1}\sqrt{(t_M r)^2+R^2}
< \sqrt{2} 
 M^{-1}R
\quad
\text{for $y\in \partial B^n_b(r)$.}
\end{equation}
Since the distance between the boundaries of the ellipsoids
(we assume that $M$ is so large that 
$\sqrt{2} /M<\lambda_1$)
$$
W_{0,(1-\sqrt{2} M^{-1}/\lambda_1)R}\subset W_{0,R}
$$
equals $\sqrt{2}M^{-1} R$, it follows from \eqref{L100} (as in \eqref{rudy102}) that
$$
W_{0,(1-\sqrt{2}M^{-1}/\lambda_1)R}\subset \tilde{f}(B^n_b(r)).
$$
Therefore
\begin{equation}
\label{L6}
\begin{split}
\H^n(\tilde{f}(B^n_b(r)))
&\geq 
\H^n(W_{0,(1-\sqrt{2} M^{-1}/\lambda_1)R})=
|J^n\tilde{f}|(0)\omega_n\left(1-\frac{\sqrt{2}}{\lambda_1 M}\right)^n R^n\\
&>
|J^n\tilde{f}|(0)\omega_n\left(1-\frac{\sqrt{2}}{\lambda_1 M}\right)^n\left(1-\frac{1}{2M}\right)r^n.
\end{split}
\end{equation}
Inequality \eqref{L3} also implies that
$$
\H^n(B^n_b(r)\setminus A)=
\H^n(B^n_b(r))-\H^n(B^n_b(r)\cap A)<\frac{1}{2M}\,\omega_n r^n.
$$
Therefore
\begin{equation}
\label{L7}
\H^n(\tilde{f}(B^n_b(r)\setminus A))\leq 
\frac{L^n}{2M}\, \omega_n r^n.
\end{equation}
We have
\begin{equation}
\label{L8}
\tilde{f}(B^n_b(r))=\tilde{f}(B^n_b(r)\cap A)\cup\tilde{f}(B^n_b(r)\setminus A)
\end{equation}
so \eqref{L8}, \eqref{L6}, and \eqref{L7} yield
\begin{equation*}
\begin{split}
\H^n(\tilde{f}(B^n_b(r)\cap A))
&\geq
\H^n(\tilde{f}(B^n_b(r)))-\H^n(\tilde{f}(B^n_b(r)\setminus A)) \\
&\geq
\omega_n r^n\left(|J^n\tilde{f}|(0)\left(1-\frac{\sqrt{2}}{\lambda_1M}\right)^n\left(1-\frac{1}{2M}\right)- \frac{L^n}{2M}\right)
\end{split}
\end{equation*}
and hence \eqref{1492} yields
\begin{equation*}
\begin{split}
\frac{\H^n_\infty(f(B^{n+m}(0,r)\cap A))}{\omega_n r^n}
&=
\frac{\H^n(\tilde{f}(B^{n+m}(0,r)\cap A))}{\omega_n r^n}\geq
\frac{\H^n(\tilde{f}(B^n_b(r)\cap A))}{\omega_n r^n}\\
&\geq
|J^n\tilde{f}|(0)\left(1-\frac{\sqrt{2}}{\lambda_1 M}\right)^n\left(1-\frac{1}{2M}\right)- \frac{L^n}{2M}\\
&\geq
|J^n\tilde{f}|(0)-\eps
\end{split}
\end{equation*}
for any $0<r<\delta$.
Therefore
$$
\Theta^n_*(f,0)=
\liminf_{r\to 0} \frac{\H^n_\infty(f(B^{n+m}(0,r)\cap A))}{\omega_n r^n}\\
\geq |J^n\tilde{f}|(0)
$$
which
completes the proof of \eqref{L10} and hence that of
Proposition~\ref{euclem}.
\end{proof}

The following example provides evidence that,
if the assumption \eqref{HCap} is replaced by 
the assumptions of Theorem~\ref{HardSard2},
then the bound \eqref{KBound} 
and global bi-Lipschitz homeomorphism $G$
cannot be recovered.
In other words, even if the $n$-density of $f$ 
satisfies $\Theta^{*n}(f,x)> 0$ on a set of positive measure,
there is no universal constant $\eta > 0$
depending only on $m$, $n$, and $\delta$
so that $\H^{n+m}(K) > \eta$.
\begin{proposition}
\label{fold}
Fix a constant $\Lambda>1$.
For any $\eps > 0$, there is a mapping 
$f:\mathbb{R}^{1+1} \supset [0,1]^2 \to \mathbb{R}$ 
with $\Theta^{*1}(f,x)=\Theta_*^1(f,x)=|J^1f|(x) = 1$ a.e.
satisfying the following:
for any measurable set $K \subset [0,1]^2$
and any $\Lambda$-bi-Lipschitz homeomorphism $G:\bbbr^2 \to \bbbr^2$
such that
$(f \circ G^{-1})|_{ (\bbbr \times \{y\}) \cap G(K) }$ is $\Lambda$-bi-Lipschitz for any $y \in \bbbr$,
we have $\H^2(K) < \eps$.
\end{proposition}
\begin{proof}
Fix $\eps > 0$ and choose $N \in \bbbn$ large enough so that $\Lambda^4 2^{1-N} \sqrt{2} < \eps$.
For any $n \in \bbbn$, 
define $f_n:[0,2^{-(n-1)}]^2 \to [0,2^{-n}]^2$ as follows:
$$
f_n(x,y) = \left\{ 
\begin{array}{ll}
(x,y) &  \text{if } (x,y) \in [0,2^{-n}] \times [0,2^{-n}] \\
(2^{-(n-1)}-x,y) &  \text{if } (x,y) \in [2^{-n},2^{-(n-1)}] \times [0,2^{-n}] \\
(x,2^{-(n-1)}-y) &  \text{if } (x,y) \in [0,2^{-n}] \times [2^{-n},2^{-(n-1)}] \\
(2^{-(n-1)}-x,2^{-(n-1)}-y) &  \text{if } (x,y) \in [2^{-n},2^{-(n-1)}] \times [2^{-n},2^{-(n-1)}] \\
\end{array}
\right. .
$$
That is, we divide $[0,2^{-(n-1)}]^2$ into four squares of equal size.
On the lower left square, $f_n$ is the identity mapping. On the upper left and lower right squares, $f_n$ is a reflection over an edge onto the lower left square. On the upper right square, $f_n$ is a reflection over both the bottom and left edges onto the lower left square.

Define $f:[0,1]^2 \to [0,2^{-N}]$ to be a composition of $N$ of these reflections 
together with the projection $\pi:\mathbb{R}^2 \to \mathbb{R}$ onto the first coordinate: $\pi(x,y) = x$.
That is, we set
$$
f := \pi \circ f_N \circ f_{N-1} \circ \cdots \circ f_2 \circ f_1
$$
Clearly, $f$ is Lipschitz.

Divide $[0,1]^2$ into $(2^N)^2$ squares $\{Q_i\}$ of side length $2^{-N}$.
Note that in each of the squares $f$ is  a composition of an isometry of $\bbbr^2$ and the orthogonal projection to $\bbbr$ so $|J^1f|=1$ and hence
$\Theta^{*1}(f,x)=\Theta_*^1(f,x)=|J^1f|(x) = 1$ a.e.

Let $G$ be any
$\Lambda$-bi-Lipschitz homeomorphism of $\bbbr^{2}$
and $K \subset [0,1]^2$ be a measurable set
such that 
$(f \circ G^{-1})|_{ (\bbbr \times \{y\}) \cap G(K) }$ is $\Lambda$-bi-Lipschitz for any $y \in \bbbr$.
Write $F = f \circ G^{-1}$.
For each $y \in \bbbr$, we have
$$
\H^1((\bbbr \times \{y\}) \cap G(K))
\leq \Lambda \H^1(F((\bbbr \times \{y\}) \cap G(K)))
\leq \Lambda \H^1([0,2^{-N}])
=\Lambda 2^{-N}.
$$
Indeed, the first inequality is a consequence of the fact that 
$F|_{ (\bbbr \times \{y\}) \cap G(K) }$ is $\Lambda$-bi-Lipschitz and the second inequality follows simply from the fact that the image of $F$ is contained in $[0,2^{-N}]$.
Note also that $\diam(G(K)) \leq \Lambda\diam(K) \leq \Lambda\sqrt{2}$.
In particular, $G(K)$ is contained in some square $Q=I_1 \times I_2$ 
where $I_1$ and $I_2$ are intervals of length $2\Lambda \sqrt{2}$.
Thus 
\begin{align*}
\H^2(K) 
\leq \Lambda^2 \H^2(G(K))=
\Lambda^2 \int_{Q} \chi_{G(K)}
&= \Lambda^2 \int_{I_2} \H^1((\bbbr \times \{y\}) \cap G(K)) \, dy \\
&\leq \Lambda^2 \int_{I_2} \Lambda 2^{-N} \, dy
=\Lambda^4 2^{1-N} \sqrt{2} < \eps .
\end{align*}
\end{proof}


\begin{thebibliography}{888}
%
\bibitem{ambrosiok}
{\sc  Ambrosio, L., Kirchheim, B.:}
Rectifiable sets in metric and Banach spaces. 
\emph{Math.\ Ann.} 318 (2000), 527--555. 
%
\bibitem{AzzSch}
{\sc Azzam, J., Schul, R.:}
Hard Sard: quantitative implicit function and extension theorems for Lipschitz maps.
\emph{Geom.\ Funct.\ Anal.} 22 (2012), 1062--1123.
%
\bibitem{davies}
{\sc Davies, R. O.:}  Increasing sequences of sets and Hausdorff measure, 
{\em Proc.\ London Math.\ Soc.} 20 (1970), 222-236.
%
\bibitem{EG}
{\sc Evans, L. C., Gariepy, R. F.:}
\emph{Measure theory and fine properties of functions.} 
Studies in Advanced Mathematics. CRC Press, Boca Raton, FL, 1992.
%
\bibitem{federer}
{\sc Federer, H.:} {\em Geometric measure theory.} Die Grundlehren der mathematischen Wissenschaften,
Band 153 Springer-Verlag New York Inc., New York 1969.
%
\bibitem{HajMal}
{\sc Haj\l{}asz, P., Malekzadeh, S.:}
On conditions for unrectifiability of a metric space.
\emph{Anal.\ Geom.\ Metr.\ Spaces} 3 (2015), 1--14.
%
\bibitem{HajMalZim}
{\sc Haj\l{}asz, P., Malekzadeh, S., Zimmerman, S.:}
Weak BLD mappings and Hausdorff measure.
\emph{Nonlinear Anal.}  177 (2018), 524-531.
%
\bibitem{HKST}
{\sc Heinonen, J., Koskela, P., Shanmugalingam, N., Tyson, J. T.:} 
\emph{Sobolev spaces on metric measure spaces. An approach based on upper gradients.} New Mathematical Monographs, 27. Cambridge University Press, Cambridge, 2015.
%
\bibitem{karmanova}
{\sc  Karmanova, M.:}
Rectifiable sets and coarea formula for metric-valued mappings. {\em J. Funct.\ Anal.} 254 (2008), 1410-1447. 
%
\bibitem{kaufman}
{\sc Kaufman, R.:} A singular map of a cube onto a square. {\em J. Differential Geom.} 14 (1979), 593-594
(1981).
%
\bibitem{MalyZ}
{\sc Mal\'y, J., Ziemer, W. P.:}
{\em Fine regularity of solutions of elliptic partial differential equations.} 
Mathematical Surveys and Monographs, 51. American Mathematical Society, Providence, RI, 1997.
%
\bibitem{mattila}
{\sc Mattila, P.:}
\emph{Geometry of sets and measures in Euclidean spaces. Fractals and rectifiability.} 
Cambridge Studies in Advanced Mathematics, 44. Cambridge University Press, Cambridge, 1995.
%
\bibitem{reichel}
{\sc Reichel, L. P.:}  The coarea formula for metric space valued maps,
Ph.D. thesis, ETH Z\"urich, 2009.
%
\bibitem{simon}
{\sc  Simon, L.:}
\emph{Lectures on geometric measure theory.} 
Proceedings of the Centre for Mathematical Analysis, Australian National University, 3. Australian National University, Centre for Mathematical Analysis, Canberra, 1983.
%
\bibitem{Whitney}
{\sc Whitney, H.:}
On totally differentiable and smooth functions.
\emph{Pacific J. Math.} 1 (1951), 143-–159.
%
\bibitem{ziemer}
{\sc Ziemer, W. P.:}
\emph{Weakly differentiable functions. Sobolev spaces and functions of bounded variation.} 
Graduate Texts in Mathematics, 120. Springer-Verlag, New York, 1989. 
\end{thebibliography}
\end{document}